\documentclass[10pt, letterpaper, reqno]{amsart}

\usepackage[a4paper,left=30mm,right=30mm,top=30mm,bottom=30mm,marginpar=30mm]{geometry}

\usepackage{amsmath, amsthm, amsfonts, amssymb, amscd, mathtools, bm} 
\usepackage{xcolor} 
\usepackage[unicode=true]{hyperref} 
\usepackage[shortlabels]{enumitem} 

\theoremstyle{plain}
\newtheorem{theorem}{Theorem}[section]
\newtheorem{lemma}[theorem]{Lemma}
\newtheorem{corollary}[theorem]{Corollary}
\newtheorem{proposition}[theorem]{Proposition}

\theoremstyle{definition}
\newtheorem{definition}[theorem]{Definition}
\newtheorem{example}[theorem]{Example}

\theoremstyle{remark}
\newtheorem{remark}[theorem]{Remark}

\numberwithin{equation}{section}

\newcommand{\ap}{\alpha_p}
\newcommand{\aq}{\alpha_q}

\newcommand{\rd}{\mathrm{d}}
\newcommand{\dm}{\mathrm{d}\mu}
\newcommand{\F}{\mathrm{F}}
\newcommand{\X}{\mathrm{X}}
\newcommand{\N}{\mathbb{N}}
\newcommand{\R}{\mathbb{R}}

\begin{document}

\title[]{\boldmath On $\F$-spaces of almost-Lebesgue functions}

\author[]{Nuno J. Alves}
\address[N. J. Alves]{
      University of Vienna, Faculty of Mathematics, Oskar-Morgenstern-Platz 1, 1090 Vienna, Austria.}
\email{nuno.januario.alves@univie.ac.at}

\begin{abstract}
We consider the space of functions almost in $L_p$ and endow it with the topology of asymptotic $L_p$-convergence. This yields a completely metrizable topological vector space which, on finite measure spaces, coincides with the space of measurable functions equipped with the topology of (local) convergence in measure. We investigate analogs of classical results such as dominated convergence and Vitali convergence theorems. For $\mathbb{R}^d$ as the underlying measure space, we establish results on approximation by smooth functions and separability. Further aspects, including local boundedness, local convexity, and duality are examined in the $\mathbb{R}^d$ setting, revealing fundamental differences from standard $L_p$ spaces.
\end{abstract}



\maketitle
\thispagestyle{empty} 

\section{Introduction}
Let $(\mathrm{X}, \Sigma, \mu)$ be a measure space. The space of real-valued measurable functions on $\X$ (identified $\mu$-a.e.) is denoted by $L_0(\X)$. Such functions are simply called \textit{measurable}. For each $p\in [1,\infty)$, the Lebesgue space of measurable functions $f$ satisfying $\| f\|_p^p = \int_{\X} |f|^p \, \rd \mu < \infty$ is denoted by $L_p(\X)$. Recall that a sequence of measurable functions $(f_n)$ is said to \textit{converge in measure} to a measurable function $f$ if for each $\delta > 0$, $\mu(|f_n-f| > \delta) \to 0$ as $n \to \infty$, and it is said to \textit{converge locally in measure} to $f$ if it converges in measure to $f$ on every measurable subset of $\X$ of finite measure. It is well known that the topology of convergence in measure on $L_0(\X)$ is metrizable, and that the topology of local convergence in measure is generated by a family of pseudometrics \cite{fremlin2003measure, bogachev2007measure}.
\par 
In this work, we introduce a complete metric vector space of measurable functions, which, on finite measure spaces, coincides with $L_0(\X)$ equipped with the topology of (local) convergence in measure. For each $p\in [1,\infty)$, we consider a space $\Lambda_p(\X)$ defined as
\begin{equation} \label{Ap}
\Lambda_p(\X) = \left\{ f \in L_0(\X) \ | \ \forall \, \delta>0 \ \exists \, E_\delta \in \Sigma \ : \ \mu(E_\delta) < \delta \ \text{and} \ f\chi_{E_\delta^c} \in L_p(\X) \right\}
\end{equation}
where $\chi_E$ denotes the characteristic function of a set $E \subseteq \X$. The space $\Lambda_p(\X)$, called \textit{almost $L_p$ space}, was initially considered in \cite{bravo2012optimal, calabuig2019representation} in a context of Banach function spaces. 
 We endow this space with a functional $\| \cdot \|_{\ap}$ defined by 
\begin{equation} \label{Fnorm}
\|f \|_{\ap} = \| \min(|f|,1)  \|_p
\end{equation}
for every measurable function $f$. For $p=1$, we shorten the notation to $\|\cdot \|_\alpha$.  This functional is known to define an extended $\F$-norm \cite{shapiro1972convexity}, and becomes finite when restricted to $\Lambda_p(\X)$, thus providing an $\F$-norm for this space; see Definition \ref{def_fnorm} and Proposition \ref{prop_fnorm} in the appendix. Consequently, $\mathrm{d}_{\ap}(f,g) \coloneq \|f - g \|_{\ap}$ defines a translation-invariant metric on $\Lambda_p(\X)$. An \textit{F-space} is a metrizable topological vector space which is complete with respect to a translation-invariant metric \cite{kalton1984space}. Our first main result is the following:
\begin{theorem} \label{mainthm1}
For each $p\in [1,\infty)$, the space $\Lambda_p(\X)$, endowed with the $\F$-norm $\| \cdot \|_{\ap}$, is an $\F$-space with the topology of asymptotic $L_p$-convergence. Moreover, if $\X$ has finite measure, then $\Lambda_p(\X)$ coincides with $L_0(\X)$, equipped with the topology of (local) convergence in measure. 
\end{theorem}

Asymptotic $L_p$-convergence was introduced in \cite{alves2024mode}, motivated by a question in diffusive relaxation. A sequence of measurable functions $(f_n)$ is said to \textit{asymptotically $L_p$-converge} (in short, \textit{$\ap$-converge}) to a measurable function $f$ if there exists a sequence of measurable sets $(B_n)$ with $\mu (B_n^c) \to 0$ as $n \to \infty$ such that 
\begin{equation} \label{alpha_conv}
\int_{B_n} |f_n - f|^p \, \mathrm{d}\mu \to 0 \quad \text{as} \ n \to \infty \, .
\end{equation}
\par

As shown in Theorem~\ref{mainthm1}, this work contributes a new example to the general theory of $\F$-spaces. Apart from $L_0(\X)$ when $\X$ has finite measure, other classical examples of $\F$-spaces include, for $0 < p < 1$, the Lebesgue spaces $L_p(\X)$, the sequence spaces $l_p$, and the Hardy spaces $H^p$ of analytic functions; see \cite{kalton1984space, rolewicz1984metric} and references therein. Since $\Lambda_p(\X)$ coincides with the space $L_0(\X)$ when $\mu(\X) < \infty$, this opens an interesting line of research aimed at understanding which properties of $\Lambda_p(\X)$, when $\mu(\X) = \infty$, align or diverge from those of $L_0(\X)$ when $\mu(\X) < \infty$. For extensive literature on $L_0(0,1)$ we refer to \cite{kalton1978endomorphisms, kalton1979quotients, kalton1981rigid, peck1981l0, kalton1984operators, kalton1985banach, faber1995lifting}. 
\par
Moreover, the space $\Lambda_p(\X)$ is complete and metrizable on any measure space, with its topology generated by an $\F$-norm. In contrast, $L_0(\X)$ with the topology of convergence in measure is also complete and metrizable, but its topology comes from a functional weaker than an $\F$-norm and involves a less practical metric. Furthermore, the topology of local convergence in measure on $L_0(\X)$ is only metrizable when the measure space is $\sigma$-finite; see Section~\ref{section_conv_measure} for details. Our construction avoids these limitations and provides a natural alternative setting for working with measurable functions on general measure spaces.

A notable feature of the spaces $\Lambda_p(\X)$ is that, although they represent all measurable functions when $\X$ has finite measure, their almost-$L_p$ structure allows for many results that have classical counterparts in the theory of Lebesgue spaces. We obtain analogs of some fundamental classical results, such as the Lebesgue dominated convergence theorem and the Vitali convergence theorem (Sections~\ref{section_dominated} and \ref{section_vitali}). We also explore the case where the underlying measure space is $\mathbb{R}^d$, establishing results on approximation by smooth functions and separability (Section~\ref{section_separable}). Furthermore, we prove that $\Lambda_p(\mathbb{R}^d)$ is neither locally bounded nor locally convex, and that its dual space is trivial, highlighting key differences from Lebesgue spaces (Section~\ref{section_local}).
\par 
The next sections are organized as follows. Section~\ref{section_preliminaries} provides preliminary material. In Section~\ref{section_metrizability}, we establish the equivalence between $\ap$-convergence and convergence with respect to the $\F$-norm $\| \cdot \|_{\ap}$. Section~\ref{section_L_vs_AL} presents some basic results on the space $\Lambda_p(\X)$; in particular, we identify the condition under which a function in $\Lambda_p(\X)$ belongs to $L_p(\X)$. Section~\ref{section_almost_conv} recalls the notion of convergence almost in $L_p$, which is closely connected to $\ap$-convergence. In Section~\ref{section_conv_measure}, we review the topology of (local) convergence in measure on $L_0(\X)$ and compare it to the topology of $\ap$-convergence on $\Lambda_p(\X)$. The proof of Theorem~\ref{mainthm1} is given in Section \ref{section_proof_thm}.

\section{Preliminaries} \label{section_preliminaries}
The present research originated from the question of whether asymptotic $L_p$-convergence can be induced by a metric. This question is motivated by the observation that metric convergence notions always yield sequential convergence classes \cite{dudley1964sequential}, and that the space of measurable functions, equipped with asymptotic $L_p$-convergence, forms such a class \cite[Theorem 4.5]{alves2024mode}. \par 
A crucial insight toward answering this question was provided in \cite{alves2024relation}, where it was proved that, on finite measure spaces, convergence in measure and asymptotic $L_p$-convergence are equivalent. Furthermore, when $\X$ has finite measure, it is known that $\mathrm{d}_{\ap}(f,g) = \|f - g \|_{\ap}$ defines a metric on $L_0(\X)$ that generates the topology of (local) convergence in measure; see \cite[Exercise 4.7.60]{bogachev2007measure} where the case $p=1$ is addressed.  

\subsection{\texorpdfstring{\boldmath Metrizability of asymptotic $L_p$-convergence}{Metrizability of asymptotic L\unichar{"209A}-convergence}} \label{section_metrizability}

~\par
For general measure spaces, we prove that the extended metric induced by the functional $ \|\cdot \|_{\ap} = \| \min(|\cdot|,1)\|_p$ generates the topology of asymptotic $L_p$-convergence. We first characterize the asymptotic $L_p$-convergence of measurable functions in terms of convergence with respect to $\| \cdot \|_{\ap}$. 
\par 
\begin{proposition}\label{prop_alpha_metric}
Let $(f_n),f$ be measurable functions. Then $(f_n)$ $\ap$-converges to $f$ if, and only if, $\|f_n - f \|_{\ap} \to 0$ as $n \to \infty$.
\end{proposition}
\begin{proof}
The result follows from the observation that for all measurable functions $f$ and $g$ one has
\begin{align*}
\int_{ \{|f-g| \leq 1 \} } |f-g|^p \, \mathrm{d}\mu + \mu(|f-g|> 1) & = \int_{\X} \min(|f-g|,1)^p \, \mathrm{d} \mu \\
&  \leq \int_B |f-g|^p \, \mathrm{d}\mu + \mu(B^c)
\end{align*}
for every measurable set $B$.
\end{proof}

Next, we characterize asymptotic $L_p$-Cauchy sequences in terms of Cauchy sequences with respect to $\| \cdot \|_{\ap}$. Recall that a sequence $(f_n)$ of measurable functions is said to be \textit{asymptotic $L_p$-Cauchy} (in short, \textit{$\ap$-Cauchy}) if there exists a sequence of measurable sets $(B_n)$ with $\mu (B_n^c) \to 0$ as $n \to \infty$ such that 
\[\int_{B_n \cap B_m} |f_n - f_m|^p \, \mathrm{d}\mu \to 0 \quad \text{as} \ n,m \to \infty \, .\]

\begin{proposition}
Let $(f_n)$ be a sequence of measurable functions. Then $(f_n)$ is $\ap$-Cauchy if, and only if, $\|f_n - f_m \|_{\ap} \to 0$ as $n,m \to \infty$.
\end{proposition}
\begin{proof}
First, suppose that $(f_n)$ is $\ap$-Cauchy and let $(B_n)$ be the sequence of measurable sets associated with this property. Then
\[\|f_n - f_m \|_{\ap}^p \leq \int_{B_n \cap B_m} |f_n - f_m|^p \, \mathrm{d}\mu + \mu(B_n^c) + \mu( B_m^c) \to 0  \quad \text{as} \ n,m \to \infty\] and hence $(f_n)$ is Cauchy with respect to $\| \cdot \|_{\ap}$. \par 
Now, assume that $(f_n)$ is Cauchy with respect to $\| \cdot \|_{\ap}$. Then 
\[\int_{\{|f_n - f_m| \leq 1 \}} |f_n-f_m|^p \, \mathrm{d}\mu \to 0  \quad \text{and} \quad \mu(|f_n - f_m|> 1) \to 0 \quad \text{as} \ n,m \to \infty.\]
We claim that $(f_n)$ is Cauchy in measure. Indeed, for $\delta > 0$ it holds that
\begin{align*}
\delta^p  \mu(|f_n - f_m| > \delta)  = \ & \delta^p \mu( \{|f_n-f_m| > \delta \} \cap \{|f_n - f_m| \leq 1  \} ) \\
 & + \delta^p \mu( \{|f_n-f_m| > \delta \} \cap \{|f_n - f_m| > 1  \} ) \\
\leq \ & \int_{\{|f_n - f_m| \leq 1 \}} |f_n-f_m|^p \, \mathrm{d}\mu + \delta^p \mu(|f_n - f_m|> 1) \to 0 \quad \text{as} \ n,m \to \infty
\end{align*}
which proves the claim. Consequently, there exists a measurable function $f$ to which $(f_n)$ converges in measure. For each $n \in \N$, let $B_n = \{|f_n - f| \leq 1/2 \}$ and note that $\mu(B_n^c) \to 0$ as $n \to \infty$. Thus, $B_n \cap B_m \subseteq \{ |f_n - f_m| \leq 1 \}$ and hence 
\[\int_{B_n \cap B_m} |f_n - f_m|^p \, \mathrm{d}\mu \leq \int_{\{|f_n - f_m| \leq 1 \}} |f_n-f_m|^p \, \mathrm{d}\mu \to 0 \quad \text{as} \ n,m \to \infty \]
which finishes the proof.
\end{proof}

\subsection{\texorpdfstring{\boldmath{Measurable functions almost in $L_p$}}{Measurable functions almost in L\unichar{"209A}}} \label{section_L_vs_AL}
~\par
In this section we deduce some basic properties of the measurable functions belonging to $\Lambda_p(\X)$. The first result provides a necessary and sufficient condition for a function in $\Lambda_p(\X)$ to also belong to $L_p(\X)$. The condition is as follows: A measurable function $f$ is said to have \textit{absolutely continuous $p$-integrals} if for each $\varepsilon > 0$ there exists $\delta_\varepsilon > 0$ such that $\int_{E} |f|^p \, \rd \mu < \varepsilon^p$ for every measurable set $E$ with $\mu(E) < \delta_\varepsilon$. 

\begin{proposition} \label{prop_L_vs_aL}
A measurable function $f$ belongs to $L_p(\X)$ if, and only if, it belongs to $\Lambda_p(\X)$ and has
absolutely continuous $p$-integrals.
\end{proposition}
\begin{proof}
It is well known that if $f$ belongs to $L_p(\X)$, then $f$ has absolutely continuous $p$-integrals. Moreover, it is clear that $L_p(\X) \subseteq \Lambda_p(\X)$. This establishes one direction.\par 
For the other direction, assume that $f$ belongs to $\Lambda_p(\X)$ and has absolutely continuous $p$-integrals. Let $\delta > 0$ be such that $\int_E |f|^p \, \rd \mu < 1 $ whenever $\mu(E) < \delta$. There exists $E_\delta$ with $\mu(E_\delta) < \delta$ such that $\int_{E_\delta^c} |f|^p \, \rd \mu < \infty.$ Then
\begin{align*}
\int_X |f|^p \, \rd \mu  = \int_{E_\delta} |f|^p \, \rd \mu + \int_{E_\delta^c} |f|^p \, \rd \mu < 1 + \int_{E_\delta^c} |f|^p \, \rd \mu  < \infty
\end{align*}
and so, $f \in L_p(\X)$.
\end{proof}

\begin{example}
The function $f : \R \to \R$ given by
\begin{align*}
f(x) = \begin{dcases}
1/x & \text{if} \ x \neq 0 \\
0 & \text{if} \ x = 0
\end{dcases}
\end{align*} 
belongs to $\Lambda_p(\R) \setminus L_p(\R)$ for every $p > 1$.
\end{example}

The next result provides a condition under which a measurable function with finite $\F$-norm $\| \cdot \|_{\ap}$ belongs to $\Lambda_p(\X)$.

\begin{proposition}
A measurable function $f$ belongs to $\Lambda_p(\X)$ if, and only if, $\|f \|_{\ap} < \infty $ and for each $\delta>0$ there exists a measurable set $E_\delta$ with $\mu(E_\delta) < \delta$ such that 
\begin{equation} \label{aux1}
\int_{E_\delta^c \cap \{ |f| > 1 \}} |f|^p \, \dm < \infty \, . 
\end{equation}
\end{proposition}
\begin{proof}
If $f \in \Lambda_p(\X)$, then $\|f \|_{\ap} < \infty$ and 
given $\delta > 0$ there exists a measurable set $E_\delta$ with $\mu(E_\delta) < \delta$ such that 
\[\int_{E_\delta^c \cap \{ |f| > 1 \}} |f|^p \, \dm \leq \int_{E_\delta^c} |f|^p \, \dm < \infty \, . \]
On the other hand, if $\|f \|_{\ap} < \infty$ and $f$ satisfies (\ref{aux1}), then given $\delta > 0$ there is $E_\delta$ with $\mu(E_\delta) < \delta$ such that 
\begin{align*}
\int_{E_\delta^c} |f|^p \, \dm & = \int_{E_\delta^c \cap \{ |f| > 1 \}} |f|^p \, \dm + \int_{E_\delta^c \cap \{ |f| \leq 1 \}} |f|^p \, \dm \\
& \leq \int_{E_\delta^c \cap \{ |f| > 1 \}} |f|^p \, \dm + \| f\|_{\ap}^p\\
& < \infty
\end{align*}
which proves that $f \in \Lambda_p(\X)$.
\end{proof}

In the last result of this section we observe that functions in $\Lambda_p(\X)$ have finite $\F$-norm $\| \cdot \|_{\aq}$ for $q \geq p$.
\begin{proposition}
If $f \in \Lambda_p(\X)$ for some $p \geq 1$, then $\|f \|_{\aq} < \infty$ for every $q \geq p$.
\end{proposition}
\begin{proof}
Let $1 \leq p \leq q$ and assume that $f \in \Lambda_p(\X)$. Then $\|f \|_{\ap}$ is finite, and since $\min(|f|,1) \leq 1$ we have $\min(|f|,1)^q \leq \min(|f|,1)^p$. Consequently,
\begin{align*}
\|f \|_{\aq}^q  = \int_{\X} \min(|f|,1)^q \, \dm \leq \int_{\X} \min(|f|,1)^p \, \dm = \|f \|_{\ap}^p  < \infty
\end{align*}
as desired.
\end{proof}

\subsection{\texorpdfstring{\boldmath Convergence almost in $L_p$}{Convergence almost in L\unichar{"209A}}} \label{section_almost_conv}

~\par
Here we recall the notion of convergence almost in $L_p$, introduced in \cite{alves2024mode}, which appears as the natural mode of convergence for the space $\Lambda_p(\X)$. A sequence $(f_n)$ of measurable functions is said to \textit{converge almost in $L_p$} to a measurable function $f$ if for each $\delta > 0$, there exists a measurable set $E_\delta$ with $\mu(E_\delta) < \delta$ such that
\[
\int_{E_\delta^c} |f_n - f|^p \, \mathrm{d}\mu \to 0 \quad \text{as } n \to \infty \, .
\]

Convergence in $L_p$ clearly implies convergence almost in $L_p$ (and $\ap$-convergence), which in turn imply (local) convergence in measure. Moreover, if $(f_n)$ converges to $f$ almost in $L_p$, then it also $\ap$-converges to $f$; conversely, if $(f_n)$ $\ap$-converges to $f$, then there exists a subsequence that converges to $f$ almost in $L_p$. For detailed proofs of these implications, we refer to \cite{alves2024mode}. Furthermore, convergence almost in $L_p$ plays an important role in establishing that $\ap$-Cauchy sequences of measurable functions are $\ap$-convergent; see \cite[Theorem 3.3]{alves2024mode}.

\par
Although convergence almost in $L_p$ may appear to be the natural mode of convergence for $\Lambda_p(\X)$, a key point in this work is that we endow $\Lambda_p(\X)$ with the topology of asymptotic $L_p$-convergence, which is slightly weaker but metrizable.

\medskip
\subsection{Topology of (local) convergence in measure} \label{section_conv_measure}
~\par

There are two classical topologies for the space $L_0(\X)$ which we recall here. The first, considered by Fr\'{e}chet \cite{frechet1921sur}, generates the topology of convergence in measure and is induced by a metric $\rd_\mu(f,g) = \|f - g \|_\mu$, where
\[
\| f \|_\mu = \min \left( \inf_{\delta > 0} \left\{ \mu(|f| > \delta) + \delta \right\}, 1 \right).
\]
Note that the truncation above is necessary to ensure that the functional $\| \cdot \|_\mu$ is finite; it is unnecessary on finite measure spaces. For instance, if $\X = [0,\infty)$, $\Sigma$ is the Borel $\sigma$-algebra, and $\mu$ is Lebesgue measure, then for the function $f(x) = x$ we have $\mu(|f| > \delta) + \delta = \infty$ for every $\delta > 0$. 

Moreover, it can be readily seen that $\| \cdot \|_\mu$ satisfies all the conditions of an $\F$-norm except that $\lim_{\lambda \to 0} \| \lambda f \|_\mu = 0$ may fail. For example, considering $f(x) = x$ as above and $\lambda_n = 1/n$, we have $\| \lambda_n f \|_\mu = 1$ for every $n \in \N$. Thus, this topology on $L_0(\X)$ is, in general, generated by a functional that is weaker than an $\F$-norm. In addition, the metric $\rd_\mu$ is somewhat impractical to work with.

\par
Now we turn our attention to the topology of local convergence in measure, which is generated by the following family of $\F$-seminorms (see Definition \ref{def_fnorm} and Corollary \ref{cor_fseminorm} in the appendix):
\[
\left\{ \| \cdot\|_{\alpha, F} \mid F \in \Sigma, \ \mu(F) < \infty \right\}
\]
where 
\[
\| f\|_{\alpha, F} \coloneq \| f \chi_F \|_{\alpha} = \int_F \min(|f|,1) \, \mathrm{d}\mu \, .
\]
Following Fremlin \cite{fremlin2003measure}, we may regard this as the standard topology on $L_0(\X)$. For $p > 1$, one could instead consider the seminorms $\| \cdot \|_{\ap,F}$, which generate the same topology since, for any $F \in \Sigma$,
\[
\| f\|_{\ap,F}^p \leq \|f \|_{\alpha,F} \leq \| f\|_{\ap,F} \, \mu(F)^{1 - \frac{1}{p}} \, .
\]
\par 
Furthermore, if $f$ is a measurable function, $F$ is a measurable set of finite measure, $p \geq 1$, and $\delta_0 > 0$, then
\begin{equation} \label{measure_implies_local}
\|f \|_{\ap,F}^p \leq \max(1,\mu(F)) \inf_{\delta>0} \left\{ \mu(|f| > \delta) + \delta \right\}
\end{equation}
and
\begin{equation} \label{alpha_implie_measure}
\inf_{\delta>0} \left\{ \mu(|f| > \delta) + \delta \right\} \leq \max(1, \delta_0^{-p}) \, \|f \|_{\ap}^p + \delta_0 \, .
\end{equation}
These estimates make precise the known implications: asymptotic $L_p$-convergence implies convergence in measure, which in turn implies local convergence in measure. Moreover, on finite measure spaces, these three modes of convergence are equivalent. On general measure spaces, a characterization of asymptotic $L_p$-convergence on $\Lambda_p(\X)$ in terms of (local) convergence in measure is given in Theorem~\ref{thm_vitali_lambda}.

\par
The topological properties of the space $L_0(\X)$ with the topology of local convergence in measure are closely linked to certain measure-theoretic properties of $(\X, \Sigma, \mu)$. We now recall some relevant notions; see \cite{fremlin2003measure}. A \textit{negligible set} is a set that is contained in a measurable set of finite measure. Moreover, a measure space $(\X, \Sigma, \mu)$ is:
\begin{enumerate}[(i)]
\item \textit{semi-finite} if for every measurable set $E \subseteq \X$ with $\mu(E) = \infty$, there exists a measurable set $F \subseteq E$ such that $0 < \mu(F) < \infty$;
\item \textit{localizable} if it is semi-finite and, for every family $\Sigma_0 \subseteq \Sigma$, there exists a set $H \in \Sigma$ such that:
\begin{itemize}
\item for every $E \in \Sigma_0$, the set $E \setminus H$ is negligible, and
\item if $G \in \Sigma$ satisfies that $E \setminus G$ is negligible for every $E \in \Sigma_0$, then $H \setminus G$ is negligible;
\end{itemize}
\item \textit{$\sigma$-finite} if there exists a sequence $(E_n)$ of measurable sets of finite measure such that $\X = \bigcup_{n \in \N} E_n$.
\end{enumerate}
Every $\sigma$-finite measure space is localizable, and therefore semi-finite; see \cite[211L]{fremlin2003measure}. The next result relates these measure-theoretic properties to the topology of local convergence in measure on $L_0(\X)$:

\begin{theorem}[{\cite[245E]{fremlin2003measure}}]
Let $(\X, \Sigma, \mu)$ be a measure space. Then $L_0(\X)$ equipped with the topology of local convergence in measure is:
\begin{enumerate}[(i)]
\item Hausdorff if, and only if, $(\X,\Sigma,\mu)$ is semi-finite,
\item Hausdorff and complete if, and only if, $(\X,\Sigma,\mu)$ is localizable,
\item metrizable if, and only if, $(\X,\Sigma,\mu)$ is $\sigma$-finite.  
\end{enumerate}
\end{theorem}
\par 
In light of Theorem~\ref{mainthm1} and the discussion above we conclude that $\Lambda_p(\X)$ is a viable alternative to $L_0(\X)$ on general measure spaces that retains key topological features. In particular, $\Lambda_p(\X)$ is metrizable (hence Hausdorff) and complete, without requiring any additional assumptions on the underlying measure space.

\section{Proof of Theorem \ref{mainthm1}} \label{section_proof_thm}

In this section we provide a proof of Theorem \ref{mainthm1} divided into three lemmas. We first evaluate the continuity of the operations of addition, $(f, g) \mapsto f + g$, and scalar multiplication, $(\lambda, f) \mapsto \lambda f$. The continuity of the addition operation follows immediately from the triangle inequality. The continuity of scalar multiplication is established in the next lemma. Consequently, $\Lambda_p(\X)$, endowed with the $\F$-norm $\|\cdot\|_{\ap}$, is a metrizable topological vector space.
\begin{lemma}
Let $(f_n),f$ belong to $\Lambda_p(\mathrm{X})$ and $(\lambda_n),\lambda$ belong to $\R$. If $(f_n)$ $\ap$-converges to $f$ and $(\lambda_n)$ converges to $\lambda$, then $(\lambda_n f_n)$ $\ap$-converges to $\lambda f $.
\end{lemma}
\begin{proof}
Let $\varepsilon > 0$. Since $f \in \Lambda_p(\mathrm{X})$, there exists a measurable set $E_\varepsilon$ with $\mu(E_\varepsilon) < \varepsilon/2^{p+1}$ such that $\int_{E_\varepsilon^c} |f|^p \, \dm = C_\varepsilon < \infty $. Moreover, from the $\ap$-convergence of $(f_n)$ towards $f$, there exist a sequence of measurable sets $(B_n)$ and $M_\varepsilon \in \N$ so that for all $n \geq M_\varepsilon$, \[\int_{B_n} |f_n-f|^p \, \dm < \frac {\varepsilon}{2^{p+1}(1+|\lambda|)^p} \quad \text{and} \quad \mu(B_n^c) < \frac{\varepsilon}{2^{p+1}} \, .\]
Additionally, let $N,K_\varepsilon \in \N$ be such that \[|\lambda_n-\lambda| < \frac{\varepsilon}{ 2^{p+1}C_\varepsilon} \quad \forall n \geq K_\varepsilon \]
and \[|\lambda_n-\lambda| < 1 \quad \forall n \geq N. \]
Set $N_\varepsilon = \max\{N,M_\varepsilon,K_\varepsilon\}$ and note that for $n \geq N_\varepsilon$, $|\lambda_n| < 1 + |\lambda|$. \par Thus, for $n \geq N_\varepsilon$ we estimate
\begin{align*}
\|\lambda_n f_n -  \lambda f \|_{\ap}^p  \leq \ & 2^{p-1} \left(\|\lambda_n f_n- \lambda_n f\|_{\ap}^p + \|\lambda_n f - \lambda f\|_{\ap}^p   \right) \\
 = \ &  2^{p-1} \left(\int_X \min(|\lambda_n||f_n-f|,1)^p \, \dm \right) \\
 &  + 2^{p-1} \left( \int_X \min(|\lambda_n-\lambda||f|,1)^p \, \dm  \right) \\
 \leq \ & 2^{p-1} \left( |\lambda_n|^p\int_{B_n}  |f_n-f|^p \, \dm +  \mu(B_n^c) \right) \\
 & + 2^{p-1} \left( |\lambda_n-\lambda|^p\int_{E_\varepsilon^c}  |f|^p \, \dm  +  \mu(E_\varepsilon) \right) \\
< \ & \varepsilon
\end{align*}
which concludes the proof.
\end{proof}
\par 
The next result establishes the completeness of $\Lambda_p(\X)$, meaning that every Cauchy sequence in $\Lambda_p(\X)$ converges to some element of $\Lambda_p(\X)$. As a consequence, we have that $\Lambda_p(\X)$ is an $\F$-space.

\begin{lemma}
Let $(f_n)$ be a sequence of functions in $\Lambda_p(\X)$. If $(f_n)$ is $\ap$-Cauchy, then there exists $f \in \Lambda_p(\mathrm{X})$ such that $ \| f_n - f\|_{\ap} \to 0 $ as $n \to \infty$.
\end{lemma}
\begin{proof}
Since $(f_n)$ is $\ap$-Cauchy, there exists a measurable function $f$ to which $(f_n)$ $\ap$-converges \cite[Theorem 3.3]{alves2024mode}. Then, there exists a subsequence $(g_n) \subseteq (f_n)$ converging to $f$ almost in $L_p$ \cite[Proposition 2.11]{alves2024mode}. Let $\delta > 0$. There exists a measurable set $E_\delta$ with $\mu(E_\delta) < \delta/2$ such that 
\[\int_{E_\delta^c} |g_n - f|^p \, \mathrm{d} \mu \to 0 \quad \text{as} \ n \to \infty \, .\] Pick $N \in \N$ so that 
\[\int_{E_\delta^c} |g_N-f|^p \, \mathrm{d}\mu < \frac{1}{2^{p-1}} \, . \] 
Since the function $g_N$ belongs to $\Lambda_p(\X),$ there exists a measurable set $F_\delta$ with $\mu(F_\delta) < \delta/2$ such that 
\[\int_{F_\delta^c} |g_N|^p \, \mathrm{d}\mu < \infty \, . \]
Set $G_\delta = E_\delta \cup F_\delta$. Then $\mu(G_\delta) < \delta$ and 
\begin{align*}
\int_{G_\delta^c} |f|^p \, \mathrm{d}\mu & \leq 2^{p-1} \int_{E_\delta^c} |g_N - f|^p \, \mathrm{d}\mu + 2^{p-1} \int_{F_\delta^c} |g_N|^p \, \mathrm{d}\mu \\
& < 1 + 2^{p-1}\int_{F_\delta^c} |g_N|^p \, \mathrm{d}\mu \\
& < \infty
\end{align*}
from which it follows that $f \in \Lambda_p(\X)$, as desired.
\end{proof}

To conclude the proof of Theorem \ref{mainthm1}, we show that on finite measure spaces, measurable functions are almost in $L_p$ for any $p \geq 1$. 

\begin{lemma} \label{lem_L0=Lambda_finite}
If $\mu(\mathrm{X}) < \infty$, then $L_0(\X) = \Lambda_p(\mathrm{X})$ for all $p \geq 1$.
\end{lemma}
\begin{proof}
We show that $L_0(\mathrm{X}) \subseteq \Lambda_p(\mathrm{X})$. Let $f \in L_0(\mathrm{X})$ and note that, by definition, $f$ is finite $\mu$-almost everywhere. For each $n \in \N$, let $E_n$ be the set $E_n = \{x \in X \ | \ |f(x)|^p > n \}$. It is clear that $E_{m} \subseteq E_n$ for all $n,m \in \N$ such that $m \geq n$. We claim that $\mu(E_n) \to 0$ as $n \to \infty$. Suppose not, then there exists $\varepsilon > 0$ such that for every $n \in \N$ there exists $\tilde{k}_n \geq n$ so that $\mu(E_{\tilde{k}_n})\geq \varepsilon$. Set $k_1 = \tilde{k}_1$. If $\tilde{k}_2 \geq k_1$, set $k_2 = \tilde{k}_2$; otherwise $k_2 = k_1$. Similarly, if $\tilde{k}_3 \geq k_2$, set $k_3 = \tilde{k}_3$; otherwise $k_3 = k_2$. Continuing this process, we find a sequence $(k_n) \to \infty$ as $n \to \infty$, such that $k_{n+1} \geq k_{n}$ and so $E_{k_{n+1}} \subseteq E_{k_{n}}$. Then \[\mu \left(\bigcap_{n\in \N} E_{k_n} \right) =  \lim_{n\to \infty} \mu(E_{k_n}) \geq \varepsilon \]
which implies that $f$ is infinite on a set of positive measure, which is impossible, thus establishing the claim. Now, let $\delta > 0$ and choose $N \in \N$ so that $\mu(E_N) < \delta$. We have
\[\int_{E_N^c} |f|^p \, \rd \mu \leq N \mu(E_N^c) < \infty \]
which establishes the result.
\end{proof}

\section{Dominated convergence results} \label{section_dominated}
The first result of this section is a Lebesgue dominated convergence theorem for the space $\Lambda_p(\X)$, being a straightforward adaptation of \cite[Theorem 5.6]{bartle1995elements}.
\begin{proposition} \label{DCT}
Let $(f_n) \subseteq \Lambda_1(\X)$ and $f$ be measurable. Assume that $(f_n)$ converges to $f$ almost everywhere and that there exists $g \in \Lambda_1(\X)$ such that $\sup_n |f_n| \leq g$ almost everywhere. Then $f \in \Lambda_1(\X)$ and for every $\delta > 0$ there exists a measurable set $E_\delta$ with $\mu(E_\delta) < \delta$ such that 
\[\int_{E_\delta^c} f_n \, \rd\mu \to \int_{E_\delta^c} f \, \rd \mu \qquad \text{as} \ n \to \infty \, .\]
\end{proposition}
\begin{proof}
Let $\delta > 0$ and $E_\delta$ be a measurable set with $\mu(E_\delta) < \delta$ such that $ g \chi_{E_\delta^c} \in L_p(\X)$. Since $|f_n| \leq g$ and $(f_n)$ converges to $f$ almost everywhere, then $|f| \leq g$ and $g+f_n$, $g - f_n$ are nonnegative almost everywhere. It follows that $f \in \Lambda_1(\X)$ and by Fatou's lemma we have:
\begin{align*}
\int_{E_\delta^c} g \, \rd \mu + \int_{E_\delta^c} f \, \rd \mu \leq \int_{E_\delta^c} g \, \rd \mu + \liminf\limits_{n\to \infty}   \int_{E_\delta^c} f_n \, \rd \mu
\end{align*}
and, similarly,
\begin{align*}
\int_{E_\delta^c} g \, \rd \mu - \int_{E_\delta^c} f \, \rd \mu & \leq  \int_{E_\delta^c} g \, \rd \mu - \limsup\limits_{n\to \infty}   \int_{E_\delta^c} f_n \, \rd \mu 
\end{align*}
whence
\[\limsup\limits_{n\to \infty}   \int_{E_\delta^c} f_n \, \rd \mu  \leq  \int_{E_\delta^c} f \, \rd \mu \leq \liminf\limits_{n\to \infty} \int_{E_\delta^c} f_n \, \rd \mu \] 
establishing the result.
\end{proof}

The next result is a consequence of the previous proposition, and the first part of its proof is very similar to that of \cite[Theorem 7.2]{bartle1995elements}.
\begin{proposition} \label{prop_DCTp}
Let $(f_n) \subseteq \Lambda_p(\X)$ and $f$ be measurable. Assume that $(f_n)$ converges to $f$ almost everywhere and that there exists $g \in \Lambda_p(\X)$ such that $\sup_n |f_n| \leq g$ almost everywhere. Then $f \in \Lambda_p(\X)$ and $\|f_n -f \|_{\ap} \to 0$ as $n \to \infty$.
\end{proposition}
\begin{proof}
Since $f_n \to f$ and $\sup_n |f_n| \leq g$ a.e., it follows that $|f| \leq g$ a.e., and hence $f \in \Lambda_p(\X)$. Moreover, we have $|f_n -f|^p \leq 2^p g^p$ a.e., which together with the facts that $\lim |f_n - f|^p = 0$ a.e. and $2^p g^p \in \Lambda_1(\X)$, implies, by Proposition \ref{DCT}, that for every $\delta > 0$ there exists $E_\delta$ with $\mu(E_\delta) < \delta$ such that \[\int_{E_\delta^c} |f_n - f|^p \, \rd \mu \to 0  \]
as $n \to \infty$, that is, $(f_n)$ converges to $f$ almost in $L_p$. \par 
Next, we provide a direct proof that convergence almost in $L_p$ implies convergence in the $\mathrm{F}$-norm $\| \cdot\|_{\ap}$; alternatively one could apply \cite[Proposition 2.9]{alves2024mode} and Proposition \ref{prop_alpha_metric}. Given $\varepsilon > 0$, there exist a measurable set $E_\varepsilon$, with $\mu(E_\varepsilon) < \varepsilon/2$, and $N_\varepsilon \in \N$ such that
\[\int_{E_\varepsilon^c} |f_n - f|^p \, \rd \mu < \frac \varepsilon 2 \qquad \forall \, n \geq N_\varepsilon \, .\]
Consequently, for every $n \geq N_\varepsilon$, 
\begin{align*}
\|f_n - f \|_{\ap}^p & = \int_X \min(|f_n-f|,1)^p \, \rd \mu \\
& = \int_{E_\varepsilon} \min(|f_n-f|,1)^p \, \rd \mu + \int_{E_\varepsilon^c} \min(|f_n-f|,1)^p \, \rd \mu \\
& \leq \mu(E_\varepsilon) + \int_{E_\varepsilon^c} |f_n-f|^p \, \rd \mu\\&  < \varepsilon
\end{align*}
as required.
\end{proof}
\par 
The previous result also holds if convergence almost everywhere is replaced by convergence in measure, serving as a counterpart to \cite[Theorem 7.8]{bartle1995elements} in the present setting.
\begin{proposition}
Let $(f_n) \subseteq \Lambda_p(\X)$ and $f$ be measurable. Assume that $(f_n)$ converges to $f$ in measure and that there exists $g \in \Lambda_p(\X)$ such that $\sup_n |f_n| \leq g$ almost everywhere. Then $f \in \Lambda_p(\X)$ and $\| f_n - f\|_{\ap} \to 0$ as $n \to \infty$.
\end{proposition}
\begin{proof}
Suppose that $(f_n)$ does not $\ap$-converge to $f$. Then, there exist $\varepsilon>0$ and a subsequence $(g_n)$ of $(f_n)$ such that $\|g_n - f \|_{\ap} \geq \varepsilon$. By hypothesis, $(g_n)$ converges to $f$ in measure, and hence there exists a subsequence $(h_n)$ of $(g_n)$ converging to $f$ almost everywhere. Given that $(h_n)$ is dominated by $g$, it follows by Proposition \ref{prop_DCTp} that $(h_n)$ converges to $f$ with respect to $\|\cdot \|_{\ap}$, a contradiction. Therefore, $(f_n)$ converges to $f$ in the $\F$-norm $\|\cdot \|_{\ap}$, which implies that $f \in A_p(\X)$.
\end{proof}

\section{Vitali convergence theorems} \label{section_vitali}

A classical Vitali convergence theorem provides necessary and sufficient conditions for a sequence in $L_p(\X)$ that converges in measure to a measurable function $f$ to also converge in $L_p(\X)$.   
\begin{theorem}[Vitali convergence theorem \cite{bartle1995elements}]\label{thm_vitali_classic} 
Let $(f_n) \subseteq L_p(\X)$ and $f$ be measurable. Then $(f_n)$ converges to $f$ in $L_p(\X)$ if, and only if,
\begin{enumerate}[(i)]
\item $(f_n)$ converges to $f$ in measure, 
\item for every $\varepsilon > 0$ there exist $E_\varepsilon \in \Sigma$ with $\mu(E_\varepsilon) < \infty$ and $\delta_\varepsilon > 0$ such that 
\[\sup_n \int_{E_\varepsilon^c} |f_n |^p \, \mathrm{d}\mu < \varepsilon^p\]
and, if $F \in \Sigma$ and $\mu(F) < \delta_\varepsilon$, then
\[\sup_n \int_{E_\varepsilon \cap F} |f_n |^p \, \mathrm{d}\mu < \varepsilon^p \, .\]
\end{enumerate}
\end{theorem}
\begin{remark}
If $\mu(X) < \infty,$ then the first part of condition $(ii)$ of Theorem \ref{thm_vitali_classic} becomes superfluous as one can simply take $E_\varepsilon = X$ for every $\varepsilon > 0$. In this case, condition $(ii)$ is replaced by \textit{uniform $p$-integrability} of the sequence $(f_n)$, that is, for every $\varepsilon>0$ there exists $\delta_\varepsilon > 0$ such that if $F \in \Sigma$ and $\mu(F) < \delta_\varepsilon$, then $\sup_n \int_{ F} |f_n |^p \, \mathrm{d}\mu < \varepsilon^p.$
\end{remark}

The next result states that for uniformly $p$-integrable sequences, $\ap$-convergence and convergence almost in $L_p$ are equivalent to $L_p$-convergence. This is analogous to the relationship between $L_p$-convergence and convergence in measure on finite measure spaces.

\begin{theorem} \label{thm_vitali_alpha}
~\par 
Let $(f_n) \subseteq L_p(\X)$ and $f$ be measurable. Then $(f_n)$ converges to $f$ in $L_p(\X)$ if, and only if,
\begin{enumerate}[(i)]
\item $(f_n)$ $\ap$-converges (or converges almost in $L_p$) to $f$, 
\item $(f_n)$ is uniformly $p$-integrable.
\end{enumerate}
\end{theorem}

\begin{proof}

First, assume that $(f_n)$ converges to $f$ in $L_p$. The condition $(i)$ is clear. Regarding condition $(ii)$, it is known that it is satisfied by $L_p$-converging sequences.  \par
To establish the converse, we show that $(f_n)$ is Cauchy in $L_p$. Since convergence almost in $L_p$ implies $\ap$-convergence, it suffices to assume the latter to establish the result. Suppose that $(f_n)$ is uniformly $p$-integrable and let $\varepsilon>0$ and $\delta_\varepsilon$ be such that $\sup_n \int_{ F} |f_n |^p \, \mathrm{d}\mu < \varepsilon^p$ whenever $\mu(F) < \delta_\varepsilon$. Assume that $(f_n)$ $\ap$-converges to $f$. Then, $(f_n)$ is $\ap$-Cauchy and hence there exists a sequence of measurable sets $(B_n)$ with $\mu(B_n^c) \to 0$ such that 
\[ \int_{B_n \cap B_m} |f_n - f_m|^p \, \mathrm{d}\mu \to 0 \qquad \text{as} \ n,m \to \infty \, . \]
Let $N_1(\varepsilon), N_2(\varepsilon) \in \mathbb{N}$ be such that 
\[\forall n \geq N_1(\varepsilon) \qquad \mu(B_n^c) < \frac{\delta_\varepsilon}{2}\]
and 
\[\forall n,m \geq N_2(\varepsilon) \qquad \int_{B_n \cap B_m} |f_n - f_m|^p \, \mathrm{d}\mu < \varepsilon^p \, . \]
Then, for $n,m \geq N_\varepsilon = \max\{N_1(\varepsilon), N_2(\varepsilon) \}$
\begin{align*}
\int_{X} |f_n - f_m|^p \, \mathrm{d}\mu & = \int_{B_n \cap B_m} |f_n - f_m|^p \, \mathrm{d}\mu + \int_{B_n^c \cup B_m^c} |f_n - f_m|^p \, \mathrm{d}\mu \\
& < \varepsilon^p + 2^{p-1} \left(\int_{B_n^c \cup B_m^c} |f_n|^p \, \mathrm{d}\mu + \int_{B_n^c \cup B_m^c} |f_m|^p \, \mathrm{d}\mu  \right).
\end{align*}
Furthermore, for $n,m \geq N_\varepsilon$ we have $\mu(B_n^c \cup B_m^c) < \delta_\varepsilon $, hence
\[\int_{B_n^c \cup B_m^c} |f_n|^p \, \mathrm{d}\mu \leq \sup_k \int_{B_n^c \cup B_m^c} |f_k|^p \, \mathrm{d}\mu < \varepsilon^p \]
and similarly for the remaining term. Consequently, for $n,m \geq N_\varepsilon$
\[\int_{X} |f_n - f_m|^p \, \mathrm{d}\mu < (1+2^p) \varepsilon^p \]
whence
\[\|f_n - f_m \|_p < (1+2^p)^{\frac{1}{p}} \, \varepsilon \]
which finishes the proof.
\end{proof}

We conclude this section with a result that serves as the counterpart of Theorem \ref{thm_vitali_classic} for the space $\Lambda_p(\X)$ endowed with the $\F$-norm $\| \cdot \|_{\ap}$.

\begin{theorem} \label{thm_vitali_lambda}
Let $(f_n) \subseteq \Lambda_p(\X)$ and $f$ be measurable. Then $(f_n)$ $\ap$-converges to $f$ if, and only if,
\begin{enumerate}[(i)]
\item $(f_n)$ converges to $f$ (locally) in measure,
\item for every $\varepsilon > 0$ there exists a measurable set $E_\varepsilon$ with $\mu(E_\varepsilon) < \infty$ such that \[\sup_n \|f_n \chi_{E_\varepsilon^c} \|_{\ap} < \varepsilon \, . \]
\end{enumerate}
\end{theorem}
\begin{proof}
First, assume that $(f_n)$ $\ap$-converges to $f$ (hence $f \in \Lambda_p(\X)$) and let $\varepsilon > 0$. Since $\ap$-convergence implies (local) convergence in measure, we put our attention on the second condition. Let $N = N(\varepsilon) \in \N$ be such that $\|f_n - f \|_{\ap} < \varepsilon/2$ whenever $n > N$. Since $f \in \Lambda_p(\X)$, there exists a measurable set $F_\varepsilon$ with $\mu(F_\varepsilon) < \varepsilon$ so that $f \chi_{F_\varepsilon^c} \in L_p(\X)$. Then, $\| f \chi_{F_\varepsilon^c \cap G_\varepsilon^c} \|_p < \varepsilon/2$ for some subset $G_\varepsilon \subseteq \X$ of finite measure. Therefore, for $n > N$, 
\begin{align*}
\| f_n \chi_{F_\varepsilon^c \cap G_\varepsilon^c} \|_{\ap} & \leq \| (f_n -f) \chi_{F_\varepsilon^c \cap G_\varepsilon^c} \|_{\ap} + \| f \chi_{F_\varepsilon^c \cap G_\varepsilon^c} \|_{\ap} \\
& \leq \| f_n -f \|_{\ap} + \| f \chi_{F_\varepsilon^c \cap G_\varepsilon^c} \|_{p} \\
& < \varepsilon \, .
\end{align*}
For each $n = 1, \ldots, N$, since $f_n \in \Lambda_p(\X)$, there exists $H_{\varepsilon,n}$ with $\mu(H_{\varepsilon,n}) < \varepsilon$ such that $f_n \chi_{H_{\varepsilon,n}^c} \in L_p(\X)$, and so $\|f_n \chi_{H_{\varepsilon,n}^c \cap I_{\varepsilon,n}^c } \|_p < \varepsilon $ for some $I_{\varepsilon,n}$ of finite measure. Define a measurable set $E_\varepsilon$ as 
\[E_\varepsilon = F_\varepsilon \cup G_\varepsilon  \cup \bigcup_{n=1}^N H_{\varepsilon,n} \cup \bigcup_{n=1}^N I_{\varepsilon,n} \]
and note that $\mu(E_\varepsilon) < \infty$. Thus, for $n > N$ we have
\begin{align*}
\| f_n \chi_{E_\varepsilon^c} \|_{\ap} \leq \|f_n \chi_{F_\varepsilon^c \cap G_\varepsilon^c} \|_{\ap} < \varepsilon
\end{align*}
and for $n = 1, \ldots, N$ it holds
\begin{align*}
\| f_n \chi_{E_\varepsilon^c} \|_{\ap} \leq \|f_n \chi_{H_{\varepsilon,n}^c \cap I_{\varepsilon,n}^c} \|_{\ap} \leq \|f_n \chi_{H_{\varepsilon,n}^c \cap I_{\varepsilon,n}^c} \|_{p} < \varepsilon
\end{align*}
establishing one direction. \par 
Now, assume that $(f_n)$ converges to $f$ locally in measure and that condition $(ii)$ holds. We prove that $(f_n)$ is $\ap$-Cauchy. Let $\varepsilon > 0$ and let $E_\varepsilon \in \Sigma$ be such that $\mu(E_\varepsilon) < \infty$ and $\|f_n \chi_{E_\varepsilon^c} \|_{\ap} < \varepsilon/4$ for each $n \in \N$. Note that, for every $n,m \in \N$,
\begin{align*}
\|f_n - f_m \|_{\ap} & \leq \|(f_n - f_m)\chi_{E_\varepsilon} \|_{\ap} +  \|f_n \chi_{E_\varepsilon^c} \|_{\ap} + \|f_m \chi_{E_\varepsilon^c} \|_{\ap}\\
 & < \|(f_n - f_m)\chi_{E_\varepsilon} \|_{\ap} + \varepsilon/2 \, .
\end{align*}
Moreover, since $\mu(E_\varepsilon) < \infty$, the local convergence in measure of $(f_n)$ towards $f$ yields the existence of $N_\varepsilon \in \N$ so that $\|f_n - f_m \|_{\alpha, E_\varepsilon} < (\varepsilon/2)^p$ for all $n,m \geq N_\varepsilon$.


 Thus, for $n,m \geq N_\varepsilon$,
\begin{align*}
\|f_n - f_m \|_{\ap} & < \|(f_n - f_m)\chi_{E_\varepsilon} \|_{\ap} + \varepsilon/2 \\
& \leq  \|f_n - f_m \|_{\alpha, E_\varepsilon}^{1/p} + \varepsilon/2 \\
& < \varepsilon 
\end{align*} 
 which completes the proof.
\end{proof}

\begin{remark}
If $\mu(\X)$ is finite, then condition $(ii)$ of Theorem \ref{thm_vitali_lambda} is automatically satisfied by setting $E_\varepsilon = \X$ for every $\varepsilon > 0$, thereby recovering the equivalence between asymptotic $L_p$-convergence and (local) convergence in measure on finite measure spaces.
\end{remark}

\begin{remark}
We observe that Theorems \ref{thm_vitali_alpha} and \ref{thm_vitali_lambda} provide a decomposition of Theorem \ref{thm_vitali_classic}. Condition \textit{(ii)} of Theorem \ref{thm_vitali_alpha} corresponds to the second part of condition \textit{(ii)} of Theorem \ref{thm_vitali_classic}, whereas condition \textit{(ii)} of Theorem \ref{thm_vitali_lambda} is the $\ap$-version of the first part of condition \textit{(ii)} of Theorem \ref{thm_vitali_classic}. 
\end{remark}

\section{Approximation and separability} \label{section_separable}
Denote by $\mathcal{S}$ the subspace of $L_0(\X)$ that comprises those functions $s : \X \to \R$ that are linear combinations of characteristic functions of measurable sets, that is, $s = \sum_{i=1}^k a_i \chi_{F_i}$, for some $a_i \in \R$ and $F_i \in \Sigma$, $i = 1, \ldots, k$. Elements of $\mathcal{S}$ are called \textit{simple functions}. \par 
The first result of this section establishes the density of simple functions in $\Lambda_p(\X)$, as a natural extension of \cite[Proposition 6.7]{folland1999real} to our setting.
\begin{proposition}
The subspace $\mathcal{S} \cap \Lambda_p(\X)$ is dense in $\Lambda_p(\X)$, that is, given $f \in \Lambda_p(\X)$ there exists a sequence $(s_n) \subseteq \Lambda_p(\X)$ of simple functions such that $\|s_n - f\|_{\ap} \to 0$ as $n \to \infty$.
\end{proposition}
\begin{proof}
Let $f \in \Lambda_p(\X)$. By applying \cite[Lemma 2.11]{bartle1995elements} to the positive and negative parts of $f$, we construct a sequence of simple functions $(s_n)$ that converges to $f$ pointwise and $\sup_n |s_n| \leq |f|$. Clearly, $(s_n)\subseteq \Lambda_p(\X)$. Additionally, $|s_n - f|^p \leq 2^p|f|^p \in \Lambda_1(\X)$ and $|s_n - f|^p \to 0$ pointwise, and hence, by Proposition \ref{DCT}, we conclude that $(s_n)$ converges to $f$ almost in $L_p$. The result follows.
\end{proof}

Note that $s \in \mathcal{S} \cap \Lambda_p(\X)$ if, and only if, for every $\delta > 0$ there exists a measurable set $E_\delta$ with $\mu(E_\delta) < \delta$ such that the set $\{s \neq 0 \} \cap E_\delta^c$ has finite measure. \par 
The next result establishes the density of $L_p(\X)$ in $\Lambda_p(\X)$. Together with Theorem \ref{mainthm1}, this implies that on finite measure spaces, $L_p(\X)$ is dense in $L_0(\X)$ with respect to the topology of (local) convergence in measure.

\begin{proposition} \label{prop_Lp_dense}
The Lebesgue space $L_p(\X)$ is dense in $\Lambda_p(\X)$, that is, given $f \in \Lambda_p(\X)$ and $\varepsilon > 0$, there exists $g_\varepsilon \in L_p(\X)$ such that $\| g_\varepsilon - f \|_{\ap} < \varepsilon$.
\end{proposition}
\begin{proof}
Let $f \in \Lambda_p(\X)$ and $\varepsilon > 0$. There exists a measurable set $E_\varepsilon$ with $\mu(E_\varepsilon) < \varepsilon^p$ such that $g_\varepsilon = f \chi_{E_\varepsilon^c}$ belongs to $L_p(\X)$. We have:
\begin{align*}
\|f - g_\varepsilon \|_{\ap}^p & = \int_{\X} \min(|f-g_\varepsilon|,1)^p \, \dm \\
& = \int_{E_\varepsilon^c} \min(|f-f \chi_{E_\varepsilon^c}|,1)^p \, \dm + \int_{E_\varepsilon} \min(|f-f \chi_{E_\varepsilon^c}|,1)^p \, \dm \, .
\end{align*}
Note that the first integral vanishes while the second is controlled from above by $\mu(E_\varepsilon)$. The result follows.
\end{proof}

Now, consider $\X = \R^d$ endowed with the Borel $\sigma$-algebra and the Lebesgue measure. We prove that $C_c^\infty(\R^d)$, the space of smooth and compactly supported functions on $\R^d$, is dense in $\Lambda_p(\R^d)$, and that $\Lambda_p(\R^d)$ is separable.
\begin{proposition} \label{prop_smooth_dense}
The space $C_c^\infty(\R^d)$ is dense in $\Lambda_p(\R^d)$, that is, given $f \in \Lambda_p(\R^d)$ and $\varepsilon > 0$, there exists $\varphi_\varepsilon \in C_c^\infty(\R^d)$ such that $\|f - \varphi_\varepsilon \|_{\ap} < \varepsilon$.
\end{proposition}
\begin{proof}
Let $f \in \Lambda_p(\R^d)$ and $\varepsilon > 0$. There exists $g_\varepsilon \in L_p(\R^d)$ such that $\| f-g_\varepsilon\|_{\ap} < \varepsilon/2$. Moreover, since $C_c^\infty(\R^d)$ is dense in $L_p(\R^d)$, there exists $\varphi_\varepsilon \in C_c^\infty(\R^d)$ such that $\| g_\varepsilon - \varphi_\varepsilon \|_{\ap} \leq \| g_\varepsilon - \varphi_\varepsilon\|_p < \varepsilon/2 $. The triangle inequality yields the desired conclusion.
\end{proof}
\begin{proposition} \label{prop_Ap_separable}
The space $\Lambda_p(\R^d)$ is separable.
\end{proposition}
\begin{proof}
Simply note that $L_p(\R^d)$ is a dense separable subspace of $\Lambda_p(\R^d)$.
\end{proof}

\section{Local boundedness, local convexity and duality} \label{section_local}
Let $V$ be real vector space with an $\F$-norm $\| \cdot \|$. The open balls $B_r = \{f \in V \ | \ \| f\| < r \}$ centered at zero form a base of neighborhoods at zero.  A subset $B$ of $V$ is called \textit{norm bounded} if $\sup \{ \|f \| \ | \ f \in B\} < \infty$, and it is called \textit{topologically bounded} if for each neighborhood $U$ of zero there exists a positive number $t$ such that $tB \subseteq U$. If $\| \cdot \|$ is a norm, then its homogeneity property yields that these two boundedness notions are equivalent. Moreover, a subset $B$ is topologically bounded if, and only if, whenever $(f_n)$ is a sequence in $B$ and $(\lambda_n)$ is a sequence of scalars converging to zero, then $\| \lambda_n f_n \|$ also converges to zero \cite{kothe1969topological}. The space $V$ is called \textit{locally bounded} if it has a topologically bounded neighborhood of zero, and it is called \textit{locally convex} if it has a base of neighborhoods at zero consisting of convex neighborhoods. We denote by $V^*$ the dual space of $V$, that is, the space of all continuous linear functionals on $V$.
\par 
In the previous sections, we observed that $\Lambda_p(\R^d)$ shares some similarities with $L_p(\R^d)$. Since $L_p(\R^d)$ is a normed space, it is locally bounded and locally convex. Moreover, its dual is nontrivial as it is isomorphic to $L_q(\R^d)$ with $q = p/(p-1)$. In stark contrast to this, we prove that the balls in $\Lambda_p(\R^d)$ centered at $0$ are neither topologically bounded nor convex, and hence $\Lambda_p(\R^d)$ is neither locally bounded nor locally convex. Furthermore, the dual space of $\Lambda_p(\R^d)$ is trivial.\par 
\begin{proposition}
The space $\Lambda_p(\R^d)$ is not locally bounded.
\end{proposition}

\begin{proof}
Let $\varepsilon>0$. For each $n \in \N$, consider an open cube $E_n \subseteq \R^d$ given by
\[E_n = \left(n, n + \left(\frac\varepsilon 2\right)^{\frac pd} \right)^d 
\]
and note that $|E_n|=(\varepsilon/2)^p$. Let $f_n = n \chi_{E_n}$. Then 
\begin{align*}
\|f_n \|_{\ap}^p = \int_{\R^d} \min(n \chi_{E_n}, 1)^p \, \rd x = \int_{E_n} \min(n, 1)^p \, \rd x = |E_n| = \left( \frac \varepsilon 2 \right)^p
\end{align*}
and so, the sequence $(f_n)$ belongs to the ball in $\Lambda_p(\R^d)$ centered at $0$ with radius $\varepsilon$. Denote this ball by $B_\varepsilon$. If $B_\varepsilon$ were (topologically) bounded, then we would have $\|a_n f_n \|_{\ap} \to 0$ as $n \to \infty$ for any sequence of scalars $(a_n)$ converging to $0$. However, taking $a_n = 1/n$ we have
\begin{align*}
\|(1/n) f_n \|_{\ap}^p = \int_{\R^d} \min( \chi_{E_n}, 1)^p \, \rd x = |E_n| = \left( \frac \varepsilon 2 \right)^p
\end{align*}
which does not converge to $0$ as $n \to \infty$. This proves that $B_\varepsilon$ is not topologically bounded, and since $\varepsilon$ is arbitrary, the result follows.  
\end{proof}

\begin{proposition}
The space $\Lambda_p(\R^d)$ is not locally convex.
\end{proposition}
\begin{proof}
Suppose, toward a contradiction, that $\Lambda_p(\R^d)$ is locally convex. Then, given $R>0$, there exists a convex neighborhood $C$ of zero contained in the ball $B_R$ of radius $R$ centered at zero. Let $0 < \varepsilon < 2$ be such that 
\[B_\varepsilon \subseteq C \subseteq B_R \]
and let $f_n = n \chi_{E_n}$, $n\in \N$, be as in the proof of the previous result. Each $f_n$ belongs to $B_\varepsilon$, and therefore to $C$ as well. Let $K \in \N$ and define $g_K$ by
\[g_K = \frac{1}{K} \sum_{n = 1}^K f_n \, . \]
Since $\varepsilon < 2$, we have $E_n \cap E_m = \emptyset$ for $n \neq m$, and hence
\begin{align*}
\|g_K \|_{\ap}^p = \sum_{n=1}^K \int_{E_n} \min(n/K,1)^p \, \rd x = \left( \frac \varepsilon2 \right)^p \frac{1}{K^p} \sum_{n=1}^K n^p \, .
\end{align*}
Choosing $K$ large enough so that $(1/K^p)\sum_{n=1}^K n^p > R^p$ yields that $g_K$ does not belong to $B_R$, and hence does not belong to $C$. This contradicts the convexity of $C$ given that $g_K$ is a convex combination of some of its elements. Consequently, $\Lambda_p(\R^d)$ is not locally convex.
\end{proof}

\begin{proposition}
The dual space of $\Lambda_p(\R^d) $ is trivial, that is, $\Lambda_p(\R^d)^* = \{0 \}$.
\end{proposition}
\begin{proof}
Assume, by contradiction, that there exists $\varphi \in \Lambda_p(\R^d)^*$ and $f \in \Lambda_p(\R^d)$ such that $\varphi(f) \neq 0$, and let $I \subsetneq \R$ be an open convex set containing zero. By linearity, $\varphi$ is surjective, and by continuity, the set $\varphi^{-1}(I)$ is open. Moreover, by linearity again, $\varphi^{-1}(I)$ is convex, and so $\varphi^{-1}(I) = \Lambda_p(\R^d)$ by the previous proposition, which contradicts the surjectivity. The result follows.
\end{proof}

\appendix
\section*{Appendix: $\F$-norms}
\setcounter{equation}{0}
\setcounter{theorem}{0}
\renewcommand\thesection{A}

We use the notions of $\F$-norm and $\F$-seminorm considered in \cite{kalton1984space, fremlin2003measure}.

\begin{definition} \label{def_fnorm}
Let $V$ be a real vector space. An $\F$-seminorm on $V$ is a functional $\| \cdot \| : V \to \R$ satisfying the following conditions:
\begin{enumerate}[(i)]
\item $\|f \| \geq 0 $ for every $f \in V$,
\vspace{1mm}
\item $\|\lambda f \| \leq \|f \|$ for each $\lambda \in \R$ with $|\lambda| \leq 1$ and every $f \in V$,
\vspace{1mm}
\item $\lim\limits_{\lambda \to 0} \|\lambda f \| = 0$ for every $f \in V$,
\vspace{1mm}
\item $\|f + g \| \leq \|f\| + \|g\|$ for every $f,g \in V$.
\end{enumerate}
An $\F$-norm $\| \cdot \|$ on $V$ is an $\F$-seminorm on $V$ such that $\|f \| > 0 $ for every $f \in V\setminus \{ 0\}$.
\end{definition}
Note that if $\|\cdot \|$ is an $\F$-norm on a real vector space $V$ then $\| f\| = 0$ if and only if $f = 0$. Additionally, an $\F$-norm $\| \cdot \|$ determines a metric $\mathrm{d}_\F$, given by $\mathrm{d}_\F(f,g) = \| f - g \|$, which is translation-invariant; that is,
$\mathrm{d}_\F(f + h, g + h) = \mathrm{d}_\F(f, g)$ for every $f,g,h \in V$. Moreover, it is clear that any norm on $V$ is an $\F$-norm. \par 

\begin{proposition} \label{prop_fnorm}
For $p\geq 1$, the functional $\| \cdot \|_{\ap} = \|\min(|\cdot|,1) \|_p$ is an $\mathrm{F}$-norm on $\Lambda_p(\mathrm{X})$.
\end{proposition}
\begin{proof}
First, we prove that $\| \cdot \|_{\ap} = \| \min(|\cdot|,1)\|_p$ is well defined on $\Lambda_p(\X)$, that is, $\| f \|_{\ap} < \infty$ for every $f \in \Lambda_p(\X)$. Given $f \in \Lambda_p(\X)$, let $E \subseteq \X$ be a measurable set with $\mu(E) < 1$ so that  $\int_{E^c} |f|^p \, \mathrm{d}\mu = C < \infty$. Then 
\begin{align*}
\|f \|_{\ap} & = \int_X \min(|f|,1)^p \, \mathrm{d}\mu \\
& \leq \int_{E^c} |f|^p \, \mathrm{d}\mu + \mu(E) \\
& < C + 1 < \infty \, .
\end{align*}
In addition, it is clear that $\| \cdot \|_{\ap}$ is nonnegative.  Regarding condition (i) of the definition, if $f$ in $\Lambda_p(\X)$ is such that $\| f\|_{\ap} = 0$, then $\min(|f|,1) = 0$ almost everywhere, which implies that $f = 0$. 
\par 
We proceed to condition (ii). Let $f \in \Lambda_p(\X)$ and $\lambda \in \R$ with $|\lambda| \leq 1$. Given $x \in \X$, if $|f(x)| < 1$, then $|\lambda f(x)| \leq 1$ and so \[\min(|\lambda f(x)|,1) = |\lambda f(x)| \leq |f(x)| = \min(|f(x)|,1) \, . \]
On the other hand, if $|f(x)| > 1$ then either $|\lambda f(x)| > 1$ or $|\lambda f(x)| \leq 1$. Assuming the former we have \[\min(|\lambda f(x)|,1) = 1 = \min(|f(x)|,1)\] while assuming the latter it holds \[\min(|\lambda f(x)|,1) = |\lambda f(x)| \leq 1 = \min(|f(x)|,1) \, .\] Applying the $L_p$ norm to both sides of the inequality $\min(|\lambda f|,1) \leq \min(|f|,1)$ yields the desired conclusion.
\par 
Next, we establish condition (iii). Let $f \in \Lambda_p(\X)$ and $\varepsilon > 0$. There exists a measurable set $E_\varepsilon$ with $\mu(E_\varepsilon) < \varepsilon^p/2$ such that $\int_{E_\varepsilon^c} |f|^p \, \dm = C_\varepsilon < \infty$. Then 
\begin{align*}
\| \lambda f \|_{\ap}^p & \leq  \int_{E_\varepsilon^c} |\lambda f|^p \, \dm + \mu(E_\varepsilon) \\
& < |\lambda|^p C_\varepsilon + \frac{\varepsilon^p}{2} \, .
\end{align*}
Choosing $\delta_\varepsilon = \varepsilon^p/(2C_\varepsilon)$ we see that for $|\lambda| < \delta_\varepsilon$ one has $\| \lambda f \|_{\ap} < \varepsilon$.
\par 
Finally, we provide a proof of the triangle inequality. Let $x \in \X$ and $f,g \in \Lambda_p(\X)$. If $\min(|f(x)|,1) < 1$ and $\min(|g(x)|,1) < 1$ then 
\begin{align*}
 \min(|f(x) + g(x)|,1)  & \leq |f(x) + g(x)| \\
 & \leq |f(x)| + |g(x)| \\
                       & = \min(|f(x) |,1) + \min(|g(x)|,1) \, .
\end{align*}
Now, suppose that either $\min(|f(x)|,1) \geq 1$ or $\min(|g(x)|,1) \geq 1$. Assume the former (the latter being similar). Then
\begin{align*}
\min(|f(x) + g(x)|,1) & \leq 1 + \min(| g(x)|,1) \\
                      & \leq \min(|f(x) |,1) + \min(|g(x)|,1) \, .
\end{align*}
Thus, the triangle inequality in $\|\cdot \|_{\ap}$ follows from the triangle inequality in $\|  \cdot \|_p$.
\end{proof}
\begin{corollary} \label{cor_fseminorm}
For $p \geq 1$ and $F \in \Sigma$ with $\mu(F) < \infty$, the functional $\| \cdot\|_{\ap,F} = \| \min(|\cdot|,1) \chi_F\|_p$ is an $\F$-seminorm on $L_0(\X)$.
\end{corollary}


\end{document}